\documentclass[12pt, notitlepage]{amsart}
\usepackage{latexsym, amsfonts, amsmath, amssymb, amsthm, cite}
\usepackage{ytableau}

\newtheorem{lemma}{Lemma}
\newtheorem{proposition}{Proposition}
\newtheorem{theorem}{Theorem}

\usepackage{graphicx}
\usepackage{mathrsfs}

  \DeclareFontFamily{U}{wncy}{}
    \DeclareFontShape{U}{wncy}{m}{n}{<->wncyr10}{}
    \DeclareSymbolFont{mcy}{U}{wncy}{m}{n}
    \DeclareMathSymbol{\Sh}{\mathord}{mcy}{"58}

\begin{document}

\title[Divisibility of character values of the symmetric group]{Almost all entries in the character table of the symmetric group are multiples of any given prime} 
\author{Sarah Peluse} 
\address{School of Mathematics, Institute for Advanced Study, Princeton, NJ 08540, USA}
\address{Department of Mathematics, Princeton University, Princeton, NJ 08540, USA}
\email{speluse@princeton.edu}

\author{Kannan Soundararajan} 
\address{Department of Mathematics, Stanford University, Stanford, CA 94305, USA}
\email{ksound@stanford.edu}

\maketitle
\begin{abstract}
We show that almost every entry in the character table of $S_N$ is divisible by any fixed prime as $N\to\infty$. This proves a conjecture of Miller.
\end{abstract}

\section{Introduction}\label{sec1}

In~\cite{Miller2019}, Miller computed the character table of $S_N$ for all $N$ up to $38$, and noticed that the proportion of entries not divisible by $2$, $3$, or $5$ seemed to tend to zero. Based on this, he conjectured that, for every fixed prime $q$, almost every entry in the character table of $S_N$ is divisible by $q$ as $N\to\infty$. It has been known for a long time, due to work of McKay~\cite{McKay1972}, that almost every character of $S_N$ has even degree. Recently, Gluck~\cite{Gluck2019} showed that the proportion of odd entries in a sparse but infinite set of columns of the character table tends to zero (see also the results of Malik, Stan, and Zaharescu~\cite{MalikStanZaharescu2014} on zeros in certain columns of the character table), and Morotti~\cite{Morotti2020} made further progress on Miller's conjecture for each fixed prime $q$. Even more recently, the first author~\cite{Peluse2020} proved Miller's conjecture for $q=2,3,5,7,11,$ and $13$.

In this paper, we completely resolve Miller's conjecture, proving it for all primes. We also give an explicit bound, uniform in $q$ almost up to $\log{N}$, for the number of entries in the character table not divisible by $q$.

\begin{theorem} 
\label{thm1}   Let $p(N)$ denote the number of unrestricted partitions of $N$, so that $p(N)^2$ denotes the number entries in the character table for $S_N$.  
Let $N$ be large, and let $q$ be a prime with $q \le (\log N)/(\log \log N)^2$.  The number of entries in the character table of $S_N$ that are not divisible by $q$ is at most 
\[
O\Big(p(N)^2 N^{-\frac{1}{12q}} \Big). 
\]
In particular, almost all entries in the character table for $S_N$ are multiples of 
\[
\prod_{q\le (\log N)/(\log \log N)^2} q. 
\]
\end{theorem} 

Miller~\cite{Miller2019} also computed the density of entries in the character table of $S_N$ divisible by $4,8,9,25,27$, and $125$, as well as the density of zeros in the character table. From this, it looks like the density of entries divisible by any fixed prime power may go to $1$, while the density of zeros may be approaching a positive constant less than $1$. Our arguments do not apply to the problem of divisibility by higher prime powers.  Regarding the number of zeros in the character table, 
Proposition~\ref{prop1} below combined with the distribution of the largest part of a random partition yields that at least a proportion $C/\log{N}$ 
(for some positive constant $C$) of the character values must be zero, and it is unclear whether a positive proportion of the entries are zero. 
However, in the related setting of finite simple groups of Lie type with rank going to infinity, Larsen and Miller~\cite{LarsenMiller2020} have shown that almost every entry of the character table is zero. 

When $\chi$ is chosen uniformly at random from the set of irreducible characters of $S_N$ and $\sigma$ is a uniformly random permutation, Miller~\cite{Miller2014} showed that $\chi(\sigma)$ almost always vanishes.  Another natural variant is to choose the character $\chi$ randomly according to the Plancherel measure (which assigns to the irreducible representation $\rho$ the weight $\text{dim} (\rho)^2/N!$).  If a conjugacy class $C$ of $S_N$ is chosen at random (with uniform measure from the $p(N)$ possibilities) then $\chi(C) =0$ almost always.  We give a brief indication of these results in \S 3. 

\medskip

\noindent {\bf Acknowledgments.}    The first author is partially supported by the NSF Mathematical Sciences Postdoctoral Research Fellowship Program under Grant No. DMS-1903038 and by the Oswald Veblen Fund.  The second author is partially supported by a grant from the National Science Foundation, and a Simons Investigator Grant from the Simons Foundation.  We thank the referees for their careful reading.  

\section{Plan of the proof}\label{sec2}
For any two partitions $\lambda$ and $\mu$ of $N$, let $\chi_\mu^{\lambda}$ denote the value of the character of $S_N$ corresponding to the partition $\lambda$ on the congruence class of permutations with cycle type $\mu$. The basic idea of the proof of Theorem~\ref{thm1} is the same as that used in~\cite{Peluse2020} to prove Miller's conjecture for $q\leq 13$: we will show that, for most partitions $\mu$ of $N$, one has $\chi_{\mu}^{\lambda}\equiv\chi_{\widetilde{\mu}}^{\lambda}\pmod{q}$ for some partition $\widetilde{\mu}$ of $N$ that possesses a part so large that $\chi_{\widetilde{\mu}}^{\lambda}$ is forced to be zero for most partitions $\lambda$ of $N$. To that end, our first proposition, which is a quantification of an argument in~\cite{Morotti2020}, states that if the partition $\mu$ of $N$ has a large part, then for most partitions $\lambda$ of $N$ one has $\chi_\mu^{\lambda}=0$.  

\begin{proposition} \label{prop1}  Let $1\le A \le \log N/\log \log N$ be a real number.  Suppose that $\mu$ is a partition of $N$ such that the largest part of $\mu$ is 
\[
\ge \frac{\sqrt{6}}{2\pi} \sqrt{N} (\log N)  \Big( 1 +  \frac 1A\Big). 
\] 
Then the number of partitions $\lambda$ of $N$ with $\chi_{\mu}^{\lambda} \neq 0$ is at most 
\[ 
O\Big(p(N) \frac{\log N}{N^{\frac 1{2A}}}\Big). 
\] 
\end{proposition}
Erd{\H o}s and Lehner~\cite{ErdosLehner1941} showed that a random partition of $N$ has largest part of size $\frac{\sqrt{6}}{2\pi} \sqrt{N} \log N + O(\sqrt{N})$, so that the partitions $\mu$ considered in Proposition \ref{prop1} are just a little bit atypical.

We will, as in~\cite{Peluse2020}, use repeated applications of the following lemma to move from our original partition $\mu$ to the partition $\widetilde{\mu}$ that we aim to show has a large part.
\begin{lemma}\label{lem1}   Let $q$ be a prime. Suppose $\mu$ is a partition of $N$, and that $\nu$ is another partition of $N$ 
obtained from $\mu$ by replacing $q$ parts of the same size $m$ by one part of size $qm$.  Then for all partitions 
$\lambda$ of $N$ we have 
\[ 
\chi_{\mu}^{\lambda} \equiv \chi_{\nu}^{\lambda} \pmod q.
\] 
\end{lemma} 

This is a simple consequence of Frobenius's formula for computing character values, see for example Section 3 of~\cite{OdlyzkoRains2000}, or Proposition 1 of~\cite{Miller2019}.  

Our second proposition says that, for a typical partition $\mu$, the partition $\widetilde{\mu}$ obtained by repeatedly applying the procedure described in Lemma~\ref{lem1} until no part appears more than $q-1$ times has a part significantly larger than $\frac{\sqrt{6}}{2\pi} \sqrt{N} \log N$.
\begin{proposition} \label{prop2}  Let $q \le (\log N)/(\log \log N)^2$ be a prime.
Starting with a partition $\mu$ of $N$, we repeatedly replace every occurrence of $q$ parts of the same size $m$ by one part of size $qm$ until we arrive 
at a partition $\widetilde{\mu}$ where no part appears more than $q-1$ times.  Then the largest part of $\widetilde{\mu}$ exceeds 
\[ 
\frac{\sqrt{6}}{2\pi} \sqrt{N}  (\log N) \Big(1 +\frac{1}{5q} \Big),
\]
except for at most 
\[
O\Big(p(N) \exp\Big(- N^{\frac{1}{15q}} \Big)\Big)
\]
 partitions $\mu$. 
\end{proposition}

If we consider only partitions of $N$ where no part appears more than $q-1$ times, then a small variation of the Erd{\H o}s--Lehner argument 
shows that such partitions typically have a largest part of size about $\frac{\sqrt{6N}}{2\pi} \frac{\sqrt{q}}{\sqrt{q-1}} \log N$.  
This suggests why a result like Proposition~\ref{prop2} may be expected.  However, some care is needed, since the partitions $\widetilde{\mu}$ 
that are the result of our procedure may not look like a typical partition with no part appearing more than $q-1$ times (for example, the largest part in $\widetilde{\mu}$ 
will very likely be a multiple of $q$).  

Theorem~\ref{thm1} is now a straightforward consequence of Propositions~\ref{prop1} and~\ref{prop2}, which we will prove in Sections~\ref{sec3} and~\ref{sec5}, respectively.
\begin{proof}[Deducing Theorem~\ref{thm1}]  We are given a prime $q\le (\log N)/(\log \log N)^2$, and wish to bound the number of 
partitions $\lambda$, $\mu$ with $\chi_{\mu}^{\lambda} \not\equiv 0 \pmod q$.   Let ${\widetilde \mu}$ be as in Proposition~\ref{prop2}.  
If the largest part of $\widetilde \mu$ is below $\frac{\sqrt{6N}}{2\pi} (\log N) (1+\frac{1}{5q} )$ then Proposition~\ref{prop2} tells us that 
there are at most 
\[ 
O\Big( p(N) \exp(-N^{\frac{1}{15q}}) \times p(N) \Big)= O\Big( p(N)^2 \exp (-N^{\frac{1}{15q}}) \Big)
\] 
choices for $\mu$ and $\lambda$.  

On the other hand, if the largest part of $\widetilde \mu$ exceeds 
$\frac{\sqrt{6N}}{2\pi} (\log N) (1+ \frac{1}{5q})$, then by Proposition~\ref{prop1} $\chi_{\widetilde \mu}^{\lambda} \neq 0$ for at most 
$O(p(N) (\log N) N^{-\frac{1}{10q}})$ partitions $\lambda$.  Thus, in this situation, since $\chi_\mu^\lambda \equiv \chi_{\widetilde \mu}^\lambda \pmod q$ by Lemma~\ref{lem1}, the number of partitions $\mu$ and $\lambda$ with $\chi_{\mu}^{\lambda} \not \equiv 0 \pmod q$ is at most 
\[ 
O\Big(p(N) \times p(N) \frac{\log N}{N^{\frac{1}{10q}}} \Big)= O\Big( p(N)^2 N^{-\frac{1}{12q}}\Big). 
\] 
Combining this (which is the bottleneck to improving Theorem~\ref{thm1} quantitatively) with our earlier bound, we conclude that there are at most 
\[ 
O\Big( p(N)^2 N^{-\frac{1}{12q}} \Big) 
\] 
pairs $\mu$, $\lambda$ with $q\nmid \chi_{\mu}^{\lambda}$. This establishes the first assertion of the theorem, and the second assertion follows upon summing this bound over all $q\le (\log N)/(\log \log N)^2$.  
\end{proof} 

\section{Proof of Proposition~\ref{prop1}}\label{sec3}

To prove Proposition~\ref{prop1}, we will need the notion of a $t$-core partition. For any box $b$ in the Young diagram of a partition, the \textit{hook-length} of $b$ is $1$ plus the number of boxes directly to the right of $b$ plus the number of boxes directly below $b$. For example, the Young diagram of $\lambda=(4,2,1)$ below has each box labeled with its hook-length.
\begin{center}
	\begin{figure}[h]
		\ytableausetup{mathmode}
		\begin{ytableau}
		6 & 4 & 2 & 1 \\
		3 & 1 \\
                1 \\
		\end{ytableau}
		\caption{Hook-lengths for $\lambda=(4,2,1)$.}
                \label{fig1}
	\end{figure}
\end{center}
A partition is called a \textit{$t$-core} if none of the hook lengths of its Young diagram are divisible by $t$. For example, from Figure~\ref{fig1} one can see that $(4,2,1)$ is a $5$-core.

\begin{proof}[Proof of Proposition~\ref{prop1}]
Let $t$ denote the largest part of $\mu$, so that $t\ge \frac{\sqrt{6N}}{2\pi} (\log N) (1+ 1/A)$ by assumption.  If the partition $\lambda$ is a $t$-core, then it follows from the Murnaghan--Nakayama rule (see Subsection 4.3 of~\cite{FultonHarris1991}) that $\chi_{\mu}^{\lambda}=0$. Now, from Lemma 5 of~\cite{Morotti2020}, we know that there are at most $(t+1) p(N-t)$ partitions $\lambda$ that are not $t$-cores.  Therefore, the number of partitions $\lambda$ with $\chi_{\mu}^{\lambda} \neq 0$ is at most 
\[
(t+1) p(N-t) \ll \frac{t+1}{N-t+1} \exp\Big( \frac{2\pi}{\sqrt{6}} \sqrt{N-t} \Big) \le 
\frac{t+1}{N-t+1} \exp\Big( \frac{2\pi}{\sqrt{6}} \sqrt{N} - \frac{\pi t}{\sqrt{6N}}\Big),
\]
where in the first inequality we have used the famous Hardy--Ramanujan asymptotic formula
$$ 
p(N) \sim \frac{1}{4N\sqrt{3}} \exp\Big( \frac{2\pi}{\sqrt{6}} \sqrt{N}\Big)
$$ 
(see \cite{Rademacher1937} for even more precise asymptotics).

 For $N\ge t\ge \frac{\sqrt{6N}}{2\pi} (\log N) (1+ 1/A)$, the right side above is maximized at the lower end point 
$t =  \frac{\sqrt{6N}}{2\pi} (\log N)(1+1/A)$, yielding the bound 
\[ 
\ll \frac{\sqrt{N}\log N}{N} N^{-\frac 12(1+\frac 1A)} \exp\Big(  \frac{2\pi}{\sqrt{6}} \sqrt{N} \Big) \ll p(N) \frac{\log N}{N^{\frac 1{2A}}}.  
\]
This establishes Proposition~\ref{prop1}.
\end{proof}

This may be a convenient juncture to elaborate on the comments at the end of our Introduction on variations of our problem.  
If the representation corresponding to $\lambda$ is chosen at random (with the uniform measure on all irreducible representations), then we have seen that 
$\lambda$ almost surely a $t$-core if $t \ge \frac{\sqrt {6N}}{2\pi} (\log N) (1+1/A)$.   A random element (chosen uniformly) $\sigma$ of the group $S_N$  
will have a cycle of length $\ge N/\log N$ with very high probability.  This is the basis of Miller's result~\cite{Miller2014} that $\chi^{\lambda}(\sigma) =0$ 
almost always.  

If the representation corresponding to $\lambda$ is chosen with the Plancherel measure, then from the work of Vershik and Kerov~\cite{VK} it 
follows that almost surely the largest part of $\lambda$ and the total number of parts in $\lambda$ are $\sim 2\sqrt{N}$, so that the maximal 
possible hook length is $\le (4+\epsilon) \sqrt{N}$.  On the other hand, by Erd{\H o}s--Lehner the largest part of a typical partition $\mu$ is 
about $\frac{\sqrt{6N}}{2\pi} \log N$.   It follows that if $\lambda$ is chosen randomly according to the Plancherel measure and the 
conjugacy class corresponding to $\mu$ is chosen uniformly, then $\chi_{\mu}^{\lambda} =0$ almost always.

\section{Preliminaries for the proof of Proposition~\ref{prop2}}\label{sec4}

Let ${\widetilde p}(j)$ denote the number of partitions of $j\geq 0$ into powers of $q$, with the convention that ${\widetilde p}(0)=1$. We define the generating function $F_q(x)$ of $\widetilde{p}(j)$ by 
\[ 
F_q(x) := \sum_{j=0}^{\infty} {\widetilde p}(j) e^{-j/x} = \prod_{j=0}^{\infty} (1- e^{-q^j/x})^{-1}
\] 
for a real number $x>0$.   Both ${\widetilde p}(j)$ and the generating function $F_q(x)$ have been studied extensively for fixed primes $q$, beginning with work of Mahler~\cite{Mahler1940} and 
de Bruijn~\cite{deBruijn1948}.  In our work we need only some simple estimates for these objects, but with uniformity in $q$.  

\begin{lemma} \label{lem2}  In the range $0 < x\le 1$, we have $F_q(x) = O(1)$.  When $x\ge 1$ we have
\begin{equation} 
\label{1} 
\frac{(\log x)^2}{2\log q} + \frac 12 \log x + O(1) \le \log F_q(x) \le \frac{(\log x)^2}{2\log q} + \frac 12 \log x+ \frac 18 \log q +O(1). 
\end{equation} 
\end{lemma} 
\begin{proof} When $x\le 1$, note that $F_q(x) \leq\prod_{j=0}^{\infty} (1- e^{-q^j})^{-1}\leq\prod_{j=0}^{\infty} (1- e^{-2^j})^{-1}$, so that $F_q(x)=O(1)$.  Now suppose $x \ge 1$.  Note that 
 the terms in the product $\prod_{j=0}^{\infty} (1- e^{-q^j/x})^{-1}$ with $q^j > x$ multiply out to a quantity bounded by $\prod_{j=0}^{\infty} (1- e^{-q^j})^{-1}$, so that they are bounded by an absolute constant.  For the terms with $q^j\le x$, 
 note that $\log (1-e^{-q^j/x})^{-1} = \log (x/q^j) + O(q^j/x)$, so that 
 \begin{align*}
 \log  \prod_{j=0}^{\infty} (1- e^{-q^j/x})^{-1} &= \sum_{0\le j\le \log x/\log q} \log\frac{x}{q^j}  + O(1) \\
& = (\log x) \Big(1+\Big\lfloor \frac{\log x}{\log q} \Big\rfloor\Big) - \frac{\log q}{2} \Big\lfloor \frac{\log x}{\log q} \Big\rfloor\Big(1+\Big\lfloor \frac{\log x}{\log q} \Big\rfloor\Big) +O(1).\\
%&\le \frac{(\log x)^2}{2\log q} + \frac 12 \log x +\frac 18 \log q +O(1), 
\end{align*}
 The estimates in \eqref{1} follow at once.  
\end{proof} 

In the second lemma of this section, we will record some basic properties of $\widetilde{p}$.
\begin{lemma} \label{lem3}  The function ${\widetilde p}(k)$ is monotone non-decreasing in $k$.   For all $r\ge  2$ we have 
\[ 
{\widetilde p}(q^r) \ge \frac{q^{r(r-1)/2}}{(r-1)^{r-1}}. 
\] 
\end{lemma} 
\begin{proof}  Appending $1$ to a partition of $k$ into powers of $q$ produces a partition of $k+1$ into powers of $q$.  This shows that 
${\widetilde p}(k)$ is monotone non-decreasing in $k$.  
 
Suppose $r\ge 2$.  For each $1\le j\le r-1$, pick a non-negative integer $k_j$ with 
$0\le k_j \le q^{r-j}/(r-1)$.  Each choice for the $k_j$'s gives a partition counted in ${\widetilde p}(q^{r})$ by using $k_j$ 
copies of $q^j$, and then using $q^r - \sum_{j=1}^{r-1} k_j q^{j}$ copies of $1$.  Therefore 
\[ 
{\widetilde p}(q^r) \ge \prod_{j=1}^{r-1} \frac{q^{r-j}}{(r-1)} =\frac{q^{r(r-1)/2}}{(r-1)^{(r-1)}}, 
\] 
as desired.
\end{proof} 

\section{Proof of Proposition~\ref{prop2}}\label{sec5}

Let us analyze the process of transforming a partition $\mu$ to a partition ${\widetilde \mu}$ as in Proposition~\ref{prop2}. Consider an integer $k$ coprime to $q$, and all parts in $\mu$ of the form $kq^j$ with $j\ge 0$.  If these parts sum to $k \ell$, then 
in the partition ${\widetilde \mu}$ we would have corresponding parts of the form $kq^j$ also summing to $k\ell$ with the additional 
property that no part appears more than $q-1$ times.  But this simply means that the number of parts of size $kq^j$ in $\widetilde{\mu}$ equals the coefficient (or `digit') of $q^j$ in the base $q$ expansion of $\ell$.  In particular, if $\ell \ge q^r$, then 
the partition $\widetilde \mu$ must have a part $kq^j$ with some $j\ge r$ (since $\ell$ must have more than $r$ digits in base $q$).

Next, suppose ${\widetilde \mu}$ has parts of the form $kq^j$ summing to $k\ell$ with no part appearing more than $q-1$ times.  From 
how many partitions $\mu$ could this $\widetilde \mu$ have arisen?  Restricting our attention to these parts of the form $kq^j$, note that $\mu$ could 
have had any collection of parts $kq^j$ that sum to $k\ell$; or in other words there are ${\widetilde p}(\ell)$ (the number of partitions of 
$\ell$ into powers of $q$) choices for parts of the form $kq^j$ in ${\mu}$.  

Let ${\mathcal K}$ be a set of integers $k$ with $(k,q)=1$.  We wish to count the number of partitions $\mu$ such that 
for $k\in {\mathcal K}$ the parts of the form $kq^j$ in $\mu$ sum to $k\ell$ with $\ell < q^r$; call this quantity $p(N;{\mathcal K}, r)$.  In other words, these are the partitions 
$\mu$ for which $\widetilde \mu$ does not have a part $kq^j$ with $j\ge r$ for all $k\in {\mathcal K}$.  By our remarks above, the count of such partitions $\mu$ is 
the coefficient of $z^N$ in the generating function 
\begin{equation} 
\label{4.1} 
\prod_{\substack { (k,q)=1 \\ k \notin {\mathcal K}}} (1+  {\widetilde p}(1) z^k + {\widetilde p}(2) z^{2k} +\ldots ) \prod_{k\in {\mathcal K}}
 \Big( \sum_{\ell=0}^{q^r-1} {\widetilde p}(\ell) z^{k\ell} \Big) = 
 \prod_{\substack { (k,q)=1 \\ k \notin {\mathcal K}}} \prod_{j=0}^{\infty} (1-z^{kq^j})^{-1}\prod_{k\in {\mathcal K}}
 \Big( \sum_{\ell=0}^{q^r-1} {\widetilde p}(\ell) z^{k\ell} \Big) . 
\end{equation}
For example, if ${\mathcal K}=\emptyset$ then we are just counting all partitions of $N$, and the above generating function may be seen to be 
$\prod_{n=1}^{\infty} (1-z^n)^{-1}$.  

Since the coefficients in the expansion of the generating function \eqref{4.1} are all positive, for any $0< z<1$ we have 
\begin{align}
\label{4.2}  
p(N;{\mathcal K},r) &\le z^{-N}  \prod_{\substack { (k,q)=1 \\ k \notin {\mathcal K}}} \prod_{j=0}^{\infty} (1-z^{kq^j})^{-1}\prod_{k\in {\mathcal K}}
 \Big( \sum_{\ell=0}^{q^r-1} {\widetilde p}(\ell) z^{k\ell} \Big) \nonumber \\
& = \Big( z^{-N} \prod_{n=1}^{\infty} (1-z^n)^{-1}\Big) \prod_{k\in {\mathcal K}}
 \Big( \frac{\sum_{\ell=0}^{q^r-1} {\widetilde p}(\ell) z^{k\ell} }{\sum_{\ell=0}^{\infty} {\widetilde p}(\ell) z^{k\ell}} \Big).
\end{align}
Here we shall take $z= e^{-1/x}$ with $x= \sqrt{N}/\sqrt{\zeta(2)} = \sqrt{6N}/\pi$.  This choice of $z$ is motivated by the fact that the asymptotic for 
the unrestricted partitions $p(N)$ arises from a contour integral computation integrating $z$ over a circle with approximately this radius.  For this choice 
of $z$, one has 
\[ 
z^{-N} \prod_{n=1}^{\infty} (1-z^n)^{-1} \sim \exp\Big( 2\sqrt{N\zeta(2)} - \frac 12 \log (\sqrt{24N})\Big) \ll  N^{\frac 34} p(N)
\] 
(see Section VIII.6 of~\cite{FlajoletSedgewick2009}). Thus, with this choice of $x$, we have
\begin{equation} 
\label{4.3} 
p(N;{\mathcal K},r) \ll N^{\frac 34} p(N) \prod_{k \in {\mathcal K}} 
\Big( \frac{\sum_{\ell=0}^{q^r-1} {\widetilde p}(\ell) e^{-k\ell/x} }{\sum_{\ell=0}^{\infty} {\widetilde p}(\ell) e^{-k\ell/x}} \Big).
\end{equation}

We are now ready for the proof of Proposition~\ref{prop2}.
\begin{proof}[Proof of Proposition~\ref{prop2}]
We apply the above considerations, taking 
\[ 
r=\Big \lfloor \frac{\log N}{2eq}\Big \rfloor 
\] 
and ${\mathcal K}$ to be the set of integers $k\ge K$ with $(k,q)=1$ where 
\[ 
K = \frac{\sqrt{6N}}{2\pi q^r} (\log N) \Big( 1 + \frac{1}{5q}\Big). 
\] 
Then $p(N;{\mathcal K}, r)$ gives an upper bound for the number of partitions $\mu$ for which $\widetilde \mu$ has largest part below 
$Kq^r= \frac{\sqrt{6N}}{2\pi} (\log N) (1 + 1/(5q))$, which is the quantity we desire to bound.   Thus \eqref{4.3} furnishes here 
the upper bound 
\begin{equation} 
\label{4.4} 
\ll N^{\frac 34} p(N) \exp\Big( - \sum_{\substack{ k\ge K \\ (k,q)=1}} \sum_{\ell \ge q^r} \frac{{\widetilde p}(\ell) e^{-\ell k/x} }{F_q(x/k)}\Big) 
= N^{\frac 34} p(N) \exp( - \Delta), 
\end{equation} 
say, for this quantity.

It remains to show that $\Delta$ is suitably large.  Since ${\widetilde p}(\ell )\ge {\widetilde p}(q^r)$ for $\ell \ge q^r$, and $F_q(x/k) \le F_q(x/K) 
\le F_q(2q^r/\log {N})$ for $k\ge K$, we obtain 
\begin{align} 
\label{4.5}
\Delta \ge \frac{{\widetilde p}(q^r) }{F_q(2q^r/\log {N})} \sum_{\substack{k \ge K\\ (k,q)=1}} \sum_{\ell \ge q^r} e^{-\ell k/x}
&= \frac{{\widetilde p}(q^r) }{F_q(2q^r/\log {N})} \sum_{\substack{k \ge K\\ (k,q)=1}} \frac{e^{-q^r k/x}}{(1-e^{-k/x})} 
\nonumber \\
&\ge \frac{{\widetilde p}(q^r) }{F_q(2q^r/\log{N})} \sum_{\substack{k \ge K\\ (k,q)=1}} e^{-q^rk/x} \frac{x}{k}. 
\end{align} 
Restricting just to the terms $K \le k \le K+x/q^r (\le 2K)$ the sum over $k$ above is 
\begin{equation} 
\label{4.6} 
\ge \frac{x}{2K} e^{-Kq^r/x -1} \sum_{\substack{ K \le k \le K+x/q^r \\ (k,q)=1}} 1 \ge 
\frac{1}{20} \frac{x}{K} \frac{x}{q^r} e^{-Kq^r/x} 
% \ge \frac{1}{20 \log N} \exp\Big(- \frac 12 (\log N)^{\frac{1}{10}} \Big), 
\ge \Big( 20 N^{\frac{1}{10q}} \log N\Big)^{-1}
\end{equation} 
after a small calculation.   Further, by Lemma~\ref{lem2} we have 
\begin{align*}
\log F_q(2q^r/\log N) &\le \frac{1}{2\log q} \Big(\log \frac{q^r}{\log \sqrt{N}}\Big)^2 + \frac 12 \log \frac{q^r}{\log \sqrt{N}} + \frac 18 \log q + O(1) \\
&\le \frac{1}{2\log q} \Big(\log \frac{q^r}{\log \sqrt{N}}\Big)^2 + \frac r2 \log q + O(1),
\end{align*} 
and combining this with the lower bound of Lemma~\ref{lem3} we obtain 
\begin{align*}
\log \frac{\widetilde{p}(q^r)}{F_q(2q^r/\log N)} 
%&\ge
% \frac{r(r-1)}{2} \log q  - r\log r \\ 
%&\hskip 0.5 in - \frac{1}{2\log q} \Big( \log \frac{q^r}{\log \sqrt{N}}\Big)^2 - \frac 12 
%\log \frac{q^r}{\log \sqrt{N}} - \frac 18 \log q  +O(1) \\ 
&\ge \frac{r^2}{2}\log q - \frac{1}{2\log q} \Big( \log \frac{q^r}{\log \sqrt{N}}\Big)^2 - r\log (qr) +O(1) \\ 
&= r\log \Big(\frac{\log \sqrt{N}}{qr}\Big) - \frac{1}{2\log q} (\log \log \sqrt{N})^2 +O(1).
\end{align*}
Using this and \eqref{4.6} in \eqref{4.5}, and recalling that $r= \lfloor (\log \sqrt{N})/(eq)\rfloor$, we conclude that 
\[ 
\Delta \gg N^{\frac{1}{2eq} - \frac{1}{10q}} \exp\Big( - \frac{1}{2\log q} (\log \log \sqrt{N})^2 -\log \log N\Big) \gg N^{\frac{1}{12q}}. 
\] 
Using this in \eqref{4.4}, we conclude that the number of partitions $\mu$ for which $\widetilde \mu$ has 
 largest part below  $\frac{\sqrt{6N}}{2\pi} (\log N) (1+ 1/(5q))$ is 
\[ 
 \ll p(N) N^{\frac 34} \exp( - \Delta) \le  p(N) \exp\big( - N^{\frac{1}{15q}}\big),  
\]
 establishing Proposition~\ref{prop2}. 
\end{proof} 
\bibliographystyle{plain}
\bibliography{bib}

\end{document}